\documentclass[12pt]{amsart}
\newtheorem{thm}{Theorem}[section]
\newtheorem{lem}[thm]{Lemma}

\newtheorem{prop}[thm]{Proposition}

\theoremstyle{remark}

\newtheorem*{prfofprop}{Proof of Proposition~\ref{Horprop}}

\newtheorem*{acknowledgement}{Acknowledgment}
\makeatletter
 
 \@addtoreset{equation}{section}

\makeatother

\title{Slope equality of non-hyperelliptic Eisenbud--Harris special fibrations of genus $4$}
\author{Makoto Enokizono}
\address{Makoto Enokizono,
	Department of Mathematics,
	Graduate School of Science,
	Osaka University,
	Toyonaka, Osaka 560-0043, Japan}
\email{m-enokizono@cr.math.sci.osaka-u.ac.jp}
\usepackage[top=35truemm,bottom=35truemm,left=25truemm,right=25truemm]{geometry}
\usepackage{amsmath,cases}
\usepackage{amsfonts}
\usepackage[all]{xy}
\usepackage{here}
\usepackage{latexsym}
  
\begin{document}
\maketitle

\begin{abstract}
The Horikawa index and the local signature are introduced for relatively minimal fibered surfaces whose general fiber is a non-hyperelliptic curve of genus $4$ with unique trigonal structure.
\end{abstract}

\section*{Introduction}

Let $S$ (resp.\ $B$) be a non-singular projective surface (resp.\ curve) defined over $\mathbb{C}$ and $f\colon S\to B$ a relatively minimal fibration whose general fiber $F$ is a non-hyperelliptic curve of genus $4$.
According to \cite{AsYo}, we say that $f$ is {\em Eisenbud--Harris special} or {\em E-H special} for short (resp.\ {\em Eisenbud--Harris general}) if $F$ has a unique $\mathfrak{g}^{1}_{3}$ (resp.\ two distinct $\mathfrak{g}^{1}_{3}$'s), or equivalently, the canonical image of $F$ lies on a quadric surface of rank $3$ (resp.\ rank $4$) in $\mathbb{P}^3$.

For E-H general fibrations of genus $4$, two important local invariants, the local signature and the Horikawa index are introduced in the appendix in \cite{AsYo}.
The purpose of this short note is to show that an analogous result also holds for E-H special fibrations of genus $4$, i.e., to show the following:

\begin{thm} \label{Horthm}
Let $\mathcal{A}$ be the set of fiber germs of relatively minimal E-H special fibrations of genus $4$.
Then the Horikawa index $\mathrm{Ind}\colon \mathcal{A}\to \mathbb{Q}_{\ge 0}$
and the local signature $\sigma\colon \mathcal{A}\to \mathbb{Q}$ are defined so that 
for any relatively minimal E-H special fibration $f\colon S\to B$ of genus $4$, 
the slope equality
$$
K_f^2=\frac{24}{7}\chi_f+\sum_{p\in B}\mathrm{Ind}(f^{-1}(p))
$$
and the localization of the signature
$$
\mathrm{Sign}(S)=\sum_{p\in B}\sigma(f^{-1}(p))
$$
hold.
\end{thm}

Note that the above slope equality was established in \cite{Tak} under the assumption that the multiplicative map $\mathrm{Sym}^{2}f_{*}\omega_f\to f_{*}\omega_f^{\otimes 2}$ is surjective, and that for non-hyperelliptic fibrations of genus $4$, the slope inequality 
$$
K_f^2\ge \frac{24}{7}\chi_f
$$
was shown independently in \cite{Ch} and \cite{kon}.

This paper is organized as follows.
In \S1, we introduce the Horikawa index and the local signature for E-H special fibrations of genus $4$.
On the other hand, for triple covering fibrations of genus $4$ over a ruled surface, which are special cases of E-H special fibrations, the Horikawa index and the local signature were introduced in \cite{Eno} and computed for any fiber germ in \cite{Eno2}.
In \S2, we prove that the Horikawa index and the local signature defined in \S1 is coincide with the ones introduced in \cite{Eno} on all fiber germs of triple covering fibrations of genus $4$ over a ruled surface.

\begin{acknowledgement}
I would like to express special thanks to Prof.\ Kazuhiro Konno for a lot of discussions and supports.
I also thank Prof.\ Tomokuni Takahashi for useful comments and discussions.
The research is supported by JSPS KAKENHI No.\ 16J00889.
\end{acknowledgement}

\section{Proof of theorem}
In this section, we prove Theorem~\ref{Horthm}.
Let $f\colon S\to B$ be a relatively minimal E-H special fibration of genus $4$.
Since the general fiber $F$ of $f$ is non-hyperelliptic, the multiplicative map $\mathrm{Sym}^{2}f_{*}\omega_f\to f_{*}\omega_f^{\otimes 2}$ is generically surjective from Noether's theorem.
Thus we have the following exact sequences of sheaves of $\mathcal{O}_B$-modules:
\begin{equation} \label{multseq}
0\to \mathcal{L}\to \mathrm{Sym}^{2}f_{*}\omega_f\to f_{*}\omega_f^{\otimes 2}\to \mathcal{T}\to 0,
\end{equation}
where the kernel $\mathcal{L}$ is a line bundle on $B$ and the cokernel $\mathcal{T}$ is a torsion sheaf on $B$.
Then the first injection defines a section $q\in H^{0}(B,\mathrm{Sym}^{2}f_{*}\omega_f\otimes \mathcal{L}^{-1})=H^{0}(\mathbb{P}_{B}(f_{*}\omega_f),2T-\pi^{*}\mathcal{L})$, where $\pi\colon \mathbb{P}_{B}(f_{*}\omega_f)\to B$ is the projection and $T=\mathcal{O}_{\mathbb{P}_{B}(f_{*}\omega_f)}(1)$ is the tautological line bundle on $\mathbb{P}_{B}(f_{*}\omega_f)$.
The section $q$ can be regarded as a relative quadratic form $q\colon (f_{*}\omega_f)^{*}\to f_{*}\omega_f\otimes \mathcal{L}^{-1}$, which defines the determinant $\mathrm{det}(q)\colon \mathrm{det}(f_{*}\omega_f)^{-1}\to \mathrm{det}(f_{*}\omega_f)\otimes \mathcal{L}^{-4}$.
Note that for a non-hyperelliptic fibration $f$  of genus $4$, $\mathrm{det}(q)=0$ if and only if $f$ is E-H special.
On the other hand, $Q=(q)\in |2T-\pi^{*}\mathcal{L}|$ is regarded as the unique relative quadric on $\mathbb{P}_{B}(f_{*}\omega_f)$ containing the image of the relative canonical map $\Phi_f\colon S\dasharrow \mathbb{P}_{B}(f_{*}\omega_f)$.
Since $f$ is E-H special, the general fiber of $\pi|_{Q}\colon Q\to B$ is a quadric of rank $3$ on $\mathbb{P}(H^{0}(F,K_F))=\mathbb{P}^3$.
The closure of the set of vertexes of general fibers of $\pi|_{Q}$ defines a section $v\colon B\to Q$, which corresponds to some quotient line bundle $\mathcal{F}$ of $f_{*}\omega_f$.
Let $\mathcal{E}$ be the kernel of the surjection $f_{*}\omega_f\to \mathcal{F}$
and put $P=\mathbb{P}_{B}(f_{*}\omega_f)$ and $P'=\mathbb{P}_{B}(\mathcal{E})$.
Let $\tau\colon \widetilde{P}\to P$ be the blow-up of $P$ along the section $v(B)$.
Then the relative projection $P\dasharrow P'$ from the section $v(B)$ extends to the morphism $\tau'\colon \widetilde{P}\to P'$
with
$$
\tau'^{*}T'=\tau^{*}T-E,
$$
where $T'=\mathcal{O}_{\mathbb{P}_{B}(\mathcal{E})}(1)$ is the tautological line bundle of $\mathbb{P}_{B}(\mathcal{E})$ and $E$ is the exceptional divisor of $\tau$.
Let $\widetilde{Q}$ denote the proper transform of $Q$ on $\widetilde{P}$.
It follows that in $\mathrm{Pic}(\widetilde{P})$,
$$
\widetilde{Q}=\tau^{*}Q-2E=\tau'^{*}(2T'-\pi'^{*}\mathcal{L}),
$$
where $\pi'\colon P'\to B$ is the projection.
Let $Q'=\tau'(\widetilde{Q})$ be the image of $\widetilde{Q}$ via $\tau'$.
It follows that $Q'\in |2T'-\pi'^{*}\mathcal{L}|$ and $\widetilde{Q}=\tau'^{*}Q'$.
The general fiber of $\pi'|_{Q'}\colon Q'\to B$ is a conic on $\mathbb{P}(H^{0}(F,\mathcal{E}|_{F}))=\mathbb{P}^2$ of rank $3$, which is isomorphic to $\mathbb{P}^1$.
Note that the composite $\tau'\circ \Phi_f\colon S\dasharrow Q'\subset P'$ of the relative canonical map $\Phi_f\colon S\dasharrow P$ and the projection $\tau'\colon P\dasharrow P'$ determines the unique trigonal structure of the general fiber $F$ of $f$.
Let $q'\in H^{0}(P',2T'-\pi'^{*}\mathcal{L})=H^{0}(B,\mathrm{Sym}^{2}\mathcal{E}\otimes \mathcal{L}^{-1})$ be a section which defines $Q'=(q')$.
Then $q'$ can be regarded as a relative quadratic form $q'\colon \mathcal{E}^{*}\to \mathcal{E}\otimes \mathcal{L}^{-1}$, which has non-zero determinant $\mathrm{det}(q')\colon \mathrm{det}(\mathcal{E})^{-1}\to \mathrm{det}(\mathcal{E})\otimes \mathcal{L}^{-3}$ since $Q'$ is of rank $3$.
Thus $\mathrm{det}(q')\in H^{0}(B,\mathrm{det}(\mathcal{E})^{\otimes 2}\otimes \mathcal{L}^{-3})$ defines an effective divisor $\Delta_{Q'}=(\mathrm{det}(q'))$ on $B$.
The degree of $\Delta_{Q'}$ is
\begin{equation} \label{disceq}
\mathrm{deg}\Delta_{Q'}=2\mathrm{deg}\mathcal{E}-3\mathrm{deg}\mathcal{L}.
\end{equation}
Let $\rho\colon \widetilde{S}\to S$ be the minimal desingularization of the rational map $\tau^{-1}\circ \Phi_f\colon S\dasharrow \widetilde{P}$ and $\widetilde{\Phi}\colon \widetilde{S}\to \widetilde{P}$ the induced morphism.
Put $\Phi=\tau\circ\widetilde{\Phi}\colon \widetilde{S}\to P$, $\Phi'=\tau'\circ\widetilde{\Phi}\colon \widetilde{S}\to P'$, $M=\Phi^{*}T$ and $M'=\Phi'^{*}T'$.
Then we can write $\rho^{*}K_f=M+Z$
for some effective vertical divisor $Z$ on $\widetilde{S}$.
Since $M'=M-\widetilde{\Phi}^{*}E$, we can also write
$\rho^{*}K_f=M'+Z'$,
where $Z'=Z+\widetilde{\Phi}^{*}E$ is also an effective vertical divisor on $\widetilde{S}$.
Since $\Phi'$ is of degree $3$ onto the image $Q'$, we have
$\Phi'_{*}\widetilde{S}=3Q'$ as cycles.
It follows that
\begin{align*}
M'^2&=(\Phi'^{*}T')^{2}\widetilde{S}=T'^{2}\Phi'_{*}\widetilde{S} \\
&=3T'^{2}Q'=3T'^{2}(2T'-\pi'^{*}\mathcal{L}) \\
&=6\mathrm{deg}\mathcal{E}-3\mathrm{deg}\mathcal{L},
\end{align*}
while we have
$$
M'^2=(\rho^{*}K_f-Z')^2=K_f^2-(\rho^{*}K_f+M')Z'.
$$
Hence we get
\begin{equation} \label{kfeq}
K_f^2=6\mathrm{deg}\mathcal{E}-3\mathrm{deg}\mathcal{L}+(\rho^{*}K_f+M')Z'.
\end{equation}
From \eqref{disceq} and \eqref{kfeq}, we can delete the term $\mathrm{deg}\mathcal{E}$ and then we have
\begin{equation} \label{degleq}
\mathrm{deg}\mathcal{L}=\frac{1}{6}K_f^2-\frac{1}{6}(\rho^{*}K_f+M')Z'-\frac{1}{2}\mathrm{deg}\Delta_{Q'}.
\end{equation}
On the other hand, taking the degree of \eqref{multseq}, we get
\begin{equation} \label{multeq}
K_f^2=4\chi_f-\mathrm{deg}\mathcal{L}+\mathrm{length}\mathcal{T}.
\end{equation}
Substituting \eqref{degleq} in the equation \eqref{multeq}, we get
$$
K_f^2=\frac{24}{7}\chi_f+\frac{1}{7}(\rho^{*}K_f+M')Z'+\frac{3}{7}\mathrm{deg}\Delta_{Q'}+\frac{6}{7}\mathrm{length}\mathcal{T}.
$$
For a fiber germ $f^{-1}(p)$, we define $\mathrm{Ind}(f^{-1}(p))$ by
$$
\mathrm{Ind}(f^{-1}(p))=\frac{1}{7}(\rho^{*}K_f+M')Z'_p+\frac{3}{7}\mathrm{mult}_p\Delta_{Q'}+\frac{6}{7}\mathrm{length}_p\mathcal{T},
$$
where $Z=\sum_{p\in B}Z_p$ is the natural decomposition with $(f\circ \rho)(Z_p)=\{p\}$ for any $p\in B$.
For the definitions of $M'$, $Z'$, etc., we do not use the completeness of the base $B$. Thus we can modify the definition of $\mathrm{Ind}$ for any fiber germ of relatively minimal E-H special fibrations of genus $4$ which is invariant under holomorphically equivalence.
Thus we can define the Horikawa index $\mathrm{Ind}\colon \mathcal{A}\to \mathbb{Q}_{\ge 0}$ such that
$$
K_f^2=\frac{24}{7}\chi_f+\sum_{p\in B}\mathrm{Ind}(f^{-1}(p)).
$$
The non-negativity of $\mathrm{Ind}(f^{-1}(p))$ is as follows.
From the nefness of $K_f$, we have $\rho^{*}K_fZ'_p\ge 0$.
For a sufficiently ample divisor $\mathfrak{a}$ on $B$, the linear system $|M'+(f\circ \rho)^{*}\mathfrak{a}|$ is free from base points. 
Thus, by Bertini's theorem, there is a smooth horizontal member $C\in |M'+(f\circ \rho)^{*}\mathfrak{a}|$ and then $M'Z'_p=(M'+(f\circ \rho)^{*}\mathfrak{a})Z'_p=CZ'_p\ge 0$.

Once the Horikawa index is introduced, we can define the local signature by the formula (cf.\ \cite{ak})
$$
\sigma(f^{-1}(p))=\frac{7}{15}\mathrm{Ind}(f^{-1}(p))-\frac{8}{15}e_f(f^{-1}(p)),
$$
where $e_f(f^{-1}(p))=e_{\mathrm{top}}(f^{-1}(p))+6$ is the Euler contribution at $p\in B$.

\section{Triple covering fibrations of genus $4$ over a ruled surface}

In this section, we consider primitive cyclic covering fibrations of type $(4,0,3)$ introduced in \cite{Eno}, which are special cases of E-H special fibrations of genus $4$.
Recall that a fibered surface $f\colon S\to B$ of genus $g$ is called a {\em primitive cyclic covering fibration of type $(g,h,n)$} if there exist a fibered surface $\varphi\colon W \to B$ of genus $h$ (not necessarily relatively minimal) and a simple $n$-cyclic covering $\pi\colon S' \to W$ branched along a smooth curve such that $f$ is the relatively minimal model of $f':=\varphi\circ \pi\colon S' \to B$.
The general fiber $F$ of a primitive cyclic covering fibration $f\colon S\to B$ of type $(4,0,3)$ is non-hyperelliptic and E-H special since it has a simple triple covering of $\mathbb{P}^1$ branched at $6$ points, which is the unique trigonal structure.
In \cite{Eno}, the Horikawa index is introduced for primitive cyclic covering fibrations of type $(g,0,n)$, which will be denoted by $\mathrm{Ind}_{g,0,n}$.
On the other hand, in the previous section, we introduced the Horikawa index $\mathrm{Ind}$ for non-hyperelliptic E-H special fibrations of genus $4$.
We show that two Horikawa indices $\mathrm{Ind}$ and $\mathrm{Ind}_{4,0,3}$ are coincide on all fiber germs of primitive cyclic covering fibrations of type $(4,0,3)$:

\begin{prop} \label{Horprop}
For any fiber germ $f\colon S\to \Delta$ of primitive cyclic covering fibrations of type $(4,0,3)$, we have
$$
\mathrm{Ind}(f^{-1}(0))=\mathrm{Ind}_{4,0,3}(f^{-1}(0)).
$$
\end{prop}

The proof is similar to the one for the coincidence of the two local signatures of hyperelliptic fibrations in the appendix in \cite{End} and plane curve fibrations in \cite{Eno3}.
We fix a fiber germ of $f\colon S\to \Delta$ of primitive cyclic covering fibrations of type $(4,0,3)$ arbitrarily.

\begin{lem} \label{cptlem}
For any integer $n>0$, there is a primitive cyclic covering fibration $\overline{f}\colon \overline{S}\to \mathbb{P}^1$ of type $(4,0,3)$ such that
$$
S\times_{\Delta}\mathrm{Spec}\mathbb{C}[t]/(t^n)\simeq \overline{S}\times_{\mathbb{P}^1}\mathrm{Spec}\mathcal{O}_{\mathbb{P}^1,0}/\mathfrak{m}^n
$$
holds over $\mathrm{Spec}\mathbb{C}[t]/(t^n)\simeq \mathrm{Spec}\mathcal{O}_{\mathbb{P}^1,0}/\mathfrak{m}^n$ and
$$
\mathrm{Ind}_{4,0,3}(\overline{f}^{-1}(t))=0
$$
holds for any $t\in \mathbb{P}^1\setminus \{0\}$.
\end{lem}

\begin{proof}
Note that the fiber germ $f\colon S\to \Delta$ is a relatively minimal model of $f'=\varphi\circ \pi \colon S'\to W\to \Delta$, where $\pi\colon S'\to W$ is a simple triple covering branched along a smooth curve $R'$ and $\varphi\colon W\to \Delta$ is obtained by a succession of blow-ups of $pr_2 \colon \mathbb{P}^1\times \Delta\to \Delta$.
Let $\psi\colon W\to \mathbb{P}^1\times \Delta$ denote the birational morphism and $R:=\psi_{*}R'$ (may be singular).
Let $\Phi(X,Y;t)$ be the defining equation of $R$ in $\mathbb{P}^1\times \Delta$, where $(X:Y)$ is the homogeneous coordinate system of $\mathbb{P}^1$ and $t$ is the coodinate of $\Delta$.
Note that $S'$ and $R'$ are obtained as ``mod $3$ resolutions'' of $\{z^3=\Phi(X,Y;t)\}$ and $R=\{\Phi(X,Y;t)=0\}$, respectively (cf.\ \cite{Eno}).
Let 
$$
\Phi(X,Y;t)=\Phi(X,Y;0)+\frac{d\Phi}{d t}(X,Y;0)t+\cdots+\frac{d^m\Phi}{d t^m}(X,Y;0)\frac{t^m}{m!}+\cdots
$$
be the Taylor expansion of $\Phi$ near $0\in \Delta$ and put
$$
\Phi^{[n]}(X,Y;t)=\Phi(X,Y;0)+\frac{d\Phi}{d t}(X,Y;0)t+\cdots+\frac{d^n\Phi}{d t^n}(X,Y;0)\frac{t^n}{n!}.
$$
Take a sufficiently large integer $m$ with $3m\gg n$ and general homogeneous polynomials $\Psi_{k}(X,Y)$, $k=n+1,\ldots,3m$ of degree $6$.
Let $\overline{\Phi}(X,Y;t_0,t_1)$ be the homogenization of 
$$
\Phi^{[n]}(X,Y;t)+\Psi_{n+1}(X,Y)t^{n+1}+\cdots+\Psi_{3m}(X,Y)t^{3m}
$$
with respect to $t\in \mathbb{C}$.
Then the curve $\overline{R}:=\{\overline{\Phi}(X,Y;t_0,t_1)=0\}\subset \mathbb{P}^1_{(X:Y)}\times \mathbb{P}^1_{(t_0:t_1)}$ satisfies
$$
R\times_{\Delta}\mathrm{Spec}\mathbb{C}[t]/(t^n)\simeq \overline{R}\times_{\mathbb{P}^1}\mathrm{Spec}\mathcal{O}_{\mathbb{P}^1,0}/\mathfrak{m}^n
$$
and smooth over $\mathbb{P}^1\setminus \{t=0\}$.
We can take a mod $3$ resolution $\overline{R}'\subset\overline{W}$ of $\overline{R}\subset \mathbb{P}^1\times \mathbb{P}^1$ by the same procedure as the one $R'$ of $R$ and a simple triple covering $\overline{\pi}\colon\overline{S}'\to \overline{W}$ branched along $\overline{R}'$ since $\overline{R}$ is divisible by $3$ in $\mathrm{Pic}(\mathbb{P}^1\times \mathbb{P}^1)$.
The relatively minimal model $\overline{f}\colon \overline{S}\to \mathbb{P}^1$ of $\overline{S}'\to \overline{W}\to \mathbb{P}^1$ is a desired one.
\end{proof}

\begin{prfofprop}
First, we assume that for the fiber germ $f\colon S\to \Delta$, the curve $R\subset \mathbb{P}^1\times \Delta$ defined in the proof of Lemma~\ref{cptlem} is smooth (according to \cite{Eno2}, the assumption is equivalent to $\mathrm{Ind}_{4,0,3}(f^{-1}(0))=0$).
From Lemma~\ref{cptlem}, we can construct a primitive cyclic covering fibration $\overline{f}\colon \overline{S}\to \mathbb{P}^1$ of type $(4,0,3)$ with
$$
\mathrm{Ind}(f^{-1}(0))=\mathrm{Ind}(\overline{f}^{-1}(0)),\quad \mathrm{Ind}_{4,0,3}(f^{-1}(0))=\mathrm{Ind}_{4,0,3}(\overline{f}^{-1}(0))=0
$$
and $\mathrm{Ind}_{4,0,3}(\overline{f}^{-1}(t))=0$ for any $t\neq 0$
since the indices $\mathrm{Ind}$ and $\mathrm{Ind}_{4,0,3}$ are algebraic invariants (see \cite{Eno3}).
Thus the slope of $f$ attains the lower bound $24/7$.
In particular, we have $\mathrm{Ind}(f^{-1}(0))=0$ from the non-negativity of $\mathrm{Ind}$.

Next, we consider a fiber germ $f\colon S\to \Delta$ arbitrarily.
Using Lemma~\ref{cptlem} again, we get a primitive cyclic covering fibration $\overline{f}\colon \overline{S}\to \mathbb{P}^1$ of type $(4,0,3)$ with
$$
\mathrm{Ind}(f^{-1}(0))=\mathrm{Ind}(\overline{f}^{-1}(0)),\quad \mathrm{Ind}_{4,0,3}(f^{-1}(0))=\mathrm{Ind}_{4,0,3}(\overline{f}^{-1}(0))
$$
and $\mathrm{Ind}(\overline{f}^{-1}(t))=\mathrm{Ind}_{4,0,3}(\overline{f}^{-1}(t))=0$
for any $t\neq 0$.
Thus we have
\begin{align*}
\mathrm{Ind}(f^{-1}(0))&=\mathrm{Ind}(\overline{f}^{-1}(0)) 
=\sum_{t\in \mathbb{P}^1}\mathrm{Ind}(\overline{f}^{-1}(t)) \\
&=K_{\overline{f}}^2-\frac{24}{7}\chi_{\overline{f}} 
=\sum_{t\in \mathbb{P}^1}\mathrm{Ind}_{4,0,3}(\overline{f}^{-1}(t)) \\
&=\mathrm{Ind}_{4,0,3}(\overline{f}^{-1}(0)) 
=\mathrm{Ind}_{4,0,3}(f^{-1}(0)).
\end{align*}

\end{prfofprop}


\begin{thebibliography}{99}

\bibitem{ak}
 T. Ashikaga and K. Konno, 
Global and local properties of pencils of algebraic curves,
Algebraic Geometry 2000 Azumino, S.\ Usui et al.\ eds, 1--49, Adv.\ Stud.\ Pure Math.\ ${\bf 36}$, Math.\ Soc.\ Japan, Tokyo, 2002.


\bibitem{AsYo}
T. Ashikaga and K. Yoshikawa, 
A divisor on the moduli space of curves associated to the signature of fibered surfaces (with an Appendix by K. Konno),
Adv.\ St.\ Pure Math.\ \textbf{56} (2009), 1--34.

\bibitem{Ch}
Z. Chen,
On the lower bound of the slope of a non-hyperelliptic fibration of genus $4$,
Intern.\ J.\ Math.\  \textbf{4} (1993), 367--378.

\bibitem{End}
H. Endo,
Meyer's signature cocycle and hyperelliptic fibrations (with an Appendix by T. Terasoma),
Math.\ Ann. \textbf{316} (2000), 237--257.

\bibitem{Eno}
 M. Enokizono, 
Slopes of fibered surfaces with a finite cyclic automorphism,
 to appear in Michigan Math.\ J.\ \textbf{66} (2017), 125--154.

\bibitem{Eno2}
 M. Enokizono, 
Fibers of cyclic covering fibrations of a ruled surface,
 to appear in Tohoku Math.\ J.

\bibitem{Eno3}
M. Enokizono,
Slope equality of plane curve fibrations and its application to Durfee's conjecture,
preprint.

\bibitem{kon}
K.~Konno,
Non-hyperelliptic fibrations of small genus and certain irregular canonical surfaces, 
Ann.\ Sc.\ Norm.\ Sup.\ Pisa ser.\ IV, \textbf{20} (1993), 575-595.

\bibitem{Tak}
T.~Takahashi,
Eisenbud-Harris special non-hyperelliptic fibrations of genus $4$,
Geom.\ Dedicata.\ \textbf{158} (2012), 191--209.






\end{thebibliography}
\end{document}